\documentclass[12pt]{amsart}
\usepackage{amssymb}

\theoremstyle{plain}
\newtheorem{theorem}{Theorem}[section]

\newtheorem{lemma}[theorem]{Lemma}

\theoremstyle{definition}

\theoremstyle{remark}
\newtheorem{example}[theorem]{Example}

\newcommand{\abs}[1]{\lvert#1\rvert}
\newcommand{\norm}[1]{\lVert#1\rVert}

\renewcommand{\le}{\leqslant}
\renewcommand{\ge}{\geqslant}
\renewcommand{\mid}{\::\:}

\newcommand{\term}[1]{{\textit{\textbf{#1}}}}

\def\range{\mathrm{range}\,}
\def\tr{\mathrm{tr}\,}

\def\bbC{\mathbb C}
\def\bbD{\mathbb D}

\def\bbN{\mathbb N}
\def\bbO{\mathbb O}

\def\bbT{\mathbb T}

\def\cG{\mathcal G}

\def\cJ{\mathcal J}

\def\cS{\mathcal S}

\def\cU{\mathcal U}
\def\cV{\mathcal V}
\def\cW{\mathcal W}
\def\cX{\mathcal X}
\def\cY{\mathcal Y}

\def\iff{\Longleftrightarrow}
\def\implies{\Longrightarrow}

\let\span\relax
\DeclareMathOperator{\span}{span}
\DeclareMathOperator{\rk}{rk}

\begin{document}
\baselineskip 18pt

\title[On matrix semigroups bounded above and below]
{On matrix semigroups bounded above and below}

\author[A.I.~Popov]{Alexey I. Popov}
\address{Department of Pure Mathematics, University of Waterloo, Waterloo, ON, N2L\,3G1. Canada}
\email{a4popov@uwaterloo.ca}
\thanks{Research supported in part by NSERC (Canada)}


\date{\today.}
\begin{abstract}
An irreducible norm closed semigroup of complex matrices is simultaneously similar to a semigroup of partial isometries if and only if (a) the norms of all nonzero members of it are uniformly bounded above and below, and (b) its idempotents commute. This is a generalization of the well-known result on bounded groups.
\end{abstract}

\maketitle


\section{Introduction}\label{section-introduction}

It is a well-known classical result, which can be found in many books on group representations (see, e.g., \cite[p.183]{NS82} or \cite[Theorem 3.1.5]{RR00}), that if $\cG$ is a bounded group of matrices, then $\cG$ is similar to a group of unitary matrices. In this paper we investigate a semigroup analogue of this result.

Recall that a matrix $T$ (or, more generally, an operator on a Hilbert space) is a partial isometry if both $T^*T$ and $TT^*$ are self-adjoint projections. The concept of a partial isometry seems a reasonable non-invertible substitute for the concept of a unitary matrix. The structure of general semigroups of partial isometries in separable Hilbert spaces was investigated in a recent study~\cite{PR} which, in particular, was a motivation for the present paper. The main result of the work~\cite{PR} is an integral representation of self-adjoint semigroups (that is, semigroups satisfying the property that $T\in\cS$ implies $T^*\in\cS$) of partial isometries acting on a separable Hilbert space. 

For the case when the underlying Hilbert space is finite-dimensional, the following structure result was obtained. It was shown that if $\cS$ is an irreducible semigroup of partial isometries (not necessarily self-adjoint) then, after a simultaneous unitary similarity, every member of $\cS$ is a block matrix with blocks coming from a fixed unitary group, such that every block row and every block column has at most one non-zero block. To state this result more precisely, let us first introduce some notation which will also be used throughout the paper.

If $n\in\bbN$ and $1\le i,j\le n$, then $E^{(n)}_{ij}$ will denote the $n\times n$ matrix whose all entries are equal to zero except for the $(i,j)$-entry which is equal to one. If $n$ is clearly understood from the context, we will use the symbol $E_{ij}$ instead of $E^{(n)}_{ij}$. If $T$ is a $k\times k$ matrix, then the matrix $E^{(n)}_{ij}\otimes T$ will denote the $kn\times kn$ matrix consisting of $k\times k$ blocks, (i.e., it is an $n\times n$ block matrix) such that the $(i,j)$-block is equal to $T$ and all the other blocks are equal to zero. Again, if $n$ is clearly understood from the context, we will use the symbol $E_{ij}\otimes T$ instead. The $k\times k$ identity matrix will be denoted by~$I_k$.

The semigroup $\cS_0^{(n)}$ is defined by $\cS_0^{(n)}=\{E_{ij}^{(n)}\mid 1\le i,j\le n\}\cup\{\bbO_n\}$ where $\bbO_n$ is the $n\times n$ zero matrix. If $\cU$ is a group of unitary $k\times k$ matrices, then we let $\cS_0^{(n)}(\cU)=\{E_{ij}^{(n)}\otimes U\mid U\in\cU,1\le i,j\le n\}\cup\{\bbO_{kn}\}$. That is, $\cS_0^{(n)}(\cU)$ is the semigroup of all the $n\times n$ block matrices such that at most one block is nonzero, and the non-zero block is a matrix from~$\cU$. The semigroup $\cS^{(n)}_1(\cU)$ is defined in a similar way, but we allow one nonzero block in each row and in each column. Again, if $n$ is understood from the context, we will use the lighter symbols $\cS_0(\cU)$ and $\cS_1(\cU)$ instead of $\cS_0^{(n)}(\cU)$ and $\cS_1^{(n)}(\cU)$. With this notation, the result from~\cite{PR} mentioned before can be stated as follows.

\begin{theorem}[\cite{PR}]\label{part-isom-semigroups}
Suppose that $\cS$ is an irreducible norm closed matrix semigroup consisting of partial isometries. Then there exists $n\in\bbN$ and an irreducible group $\cU$ of unitary matrices such that, after a unitary similarity,
$$
\cS^{(n)}_0(\cU)\subseteq \cS\subseteq \cS^{(n)}_1(\cU).
$$
The size of matrices in $\cU$ is equal to the minimal non-zero rank of operators in~$\cS$.
\end{theorem}

In this paper, we characterize the irreducible norm closed matrix semigroups similar to semigroups of partial isometries (Theorem~\ref{main-result}). We obtain  conditions on the semigroup in the same spirit as in the result about groups mentioned in the first paragraph which make sure that the semigroup is similar to a semigroup of partial isometries. Namely, we show that an irreducible semigroup is similar to a semigroup of partial isometries if and only if its nonzero elements are uniformly bounded in norm above and below and its idempotent elements commute. This turns out also to be equivalent to the condition that the spectrum of every element in the semigroup is contained in $\{0\}\cup\bbT$ and the idempotents in the semigroup commute. Notice that all these conditions are trivially satisfied for bounded groups of matrices. 

Our proofs work for complex matrices. Throughout the paper, we consider $\bbC^n$ as an inner product space, with the standard inner product. In particular, the norm in $\bbC^n$, denoted by $\norm{\cdot}$, is the norm defined by this inner product.

The symbol $\bbD$ will denote the unit disk of $\bbC$ and $\bbT$ will denote the unit circle.  If $T$ is an $n\times n$ matrix, then $\sigma(T)$ will denote the spectrum of $T$ and $r(T)$ will denote the spectral radius of~$T$. Whenever we mention a norm of a matrix, we always mean its operator norm. That is, $\norm{T}=\sup\{\norm{Tx}\mid x\in\bbC^n,\ \norm{x}\le 1\}$. If $X$ and $Y$ are subspaces of $\bbC^n$, then the symbol $X\oplus Y$ denotes the direct sum of $X$ and $Y$. In particular, this means that $X$ and $Y$ are orthogonal to each other. The symbol $X\ominus Y$ means the space $\{x\in X\mid x\perp Y\}$. The orthogonal complement of a space $X\subseteq\bbC^n$ is denoted by $X^\perp$. A matrix $P$ with the property that $P=P^2$ will be referred as idempotent. An idempotent $P$ that, in addition, satisfies $P=P^*$, will be called a projection.

\term{Acknowledgment}. The author would like to thank Heydar Radjavi for numerous useful and stimulating discussions and the anonymous referee for pointing out the problem demonstrated in Example~\ref{closure-important}.


\section{Similarity to semigroups of partial isometries}\label{section-one}

This section contains our main result (Theorem~\ref{main-result}) which characterizes the irreducible matrix semigroups similar to semigroups of partial isometries. We will need a number of lemmas in our proof.

The following two lemmas are well-known and are very useful when working with bounded irreducible matrix semigroups. Their proofs use manipulations with Jordan forms and can be found in \cite[pp.~48-49]{RR00}.

\begin{lemma}\label{bounded-powers-unitary}
Let $T$ be a matrix such that $\sigma(T)\subseteq\bbT$ and the set $\{\norm{T^n}\mid n\in\bbN\}$ is bounded. Then $T$ is similar to a unitary diagonal matrix.
\end{lemma}

\begin{lemma}\label{idempotents}
Let $\cS$ be a bounded norm closed matrix semigroup and $T\in\cS$ be such that $r(T)=1$. Then $\cS$ contains an idempotent $P$ such that $\range(P)\subseteq\range(T)$ and $\ker(P)\supseteq\ker(T)$.
\end{lemma}

The following lemma is standard.

\begin{lemma}\label{irreducibility-criterion}
Let $\cS$ be a matrix semigroup. Then $\cS$ is irreducible if, and only if, for every pair $A,B$ of nonzero matrices, the set $A\cS B=\{ATB\mid T\in\cS\}$ contains a nonzero member.
\end{lemma}

The next lemma was proved in \cite[Lemma~3.1]{OR97}. The assumptions of \cite[Lemma~3.1]{OR97} are slightly stronger than what we have (namely, $n_0$ is assumed to be equal to one); however, exactly the same proof works for the general~$n_0$.

\begin{lemma}\label{polynomials-zero}
Suppose that $n_0\in\bbN$, $p_1,p_2,\dots,p_r$ are polynomials and $\alpha_1,\alpha_2,\dots,\alpha_r$ are distinct nonzero complex numbers such that
$$
\sum_{i=1}^rp_i(n)\alpha_i^n=0
$$
holds for all $n\ge n_0$. Then all the polynomials $p_1,p_2,\dots,p_r$ are identically zero.
\end{lemma}

The following theorem is the main result of this paper.

\begin{theorem}\label{main-result}
Let $\cS$ be an irreducible norm closed semigroup of complex $n\times n$ matrices. Then the following conditions are equivalent.
\begin{enumerate}
\item $\cS$ is simultaneously similar to a semigroup of partial isometries;
\item the idempotent members of $\cS$ all commute and $\sigma(T)\subseteq\{0\}\cup\bbT$ for all $T\in\cS$;
\item the idempotent members of $\cS$ all commute and there exist $c_1,c_2>0$ such that
$$
c_1\le\norm{T}\le c_2
$$
for all nonzero $T\in\cS$.
\end{enumerate}
\end{theorem}

\begin{proof}
%
Let us first show the implication \makebox{(i)$\Longrightarrow$(ii)}. If $\cS$ is simultaneously similar to a semigroup of partial isometries, then, by Theorem~\ref{part-isom-semigroups}, after a simultaneous similarity every member of $\cS$ is a block matrix with unitary blocks, such that each block row and each block column has at most one nonzero block. Obviously, every idempotent in~$\cS$ is a block-diagonal matrix, such that the diagonal blocks are equal to either zero or the identity matrix. It is easy to see that all such idempotents commute. Also, a power of every matrix $T$ from such a semigroup is a block-diagonal matrix with blocks from $\cU\cup\{0\}$, so that $\sigma(T)\subseteq\{0\}\cup\bbT$. This shows that \makebox{(i)$\Longrightarrow$(ii)}.

Next, we will show that \makebox{(ii)$\Longrightarrow$(i)} holds. Suppose that $\cS$ is a semigroup such that all the idempotents in $\cS$ commute and $\sigma(T)\subseteq\{0\}\cup\bbT$ for every matrix $T\in\cS$. Observe that the trace functional is bounded on $\cS$; in fact, $\abs{\tr(T)}\le n$ for every $T\in\cS$. Since $\cS$ is irreducible, it follows from \cite[Proposition 4.9]{Okn98} (see also \cite[Theorem 1]{RR08}) that $\cS$ is bounded. 

Pick a maximal family $\{P_1,\dots,P_m\}$ of disjoint idempotents in~$\cS$. Clearly, each $P_i$ is minimal in the sense that no other nonzero idempotent in $\cS$ has range contained in that of~$P_i$. Applying \cite[Lemma 3.1.6(iii)]{RR00} to $P_i\cS P_i$, we may assume that the rank of each $P_i$ is equal to the minimal nonzero rank of matrices in~$\cS$ which we will denote by~$k$ (we warn the reader that this statement in \cite{RR00} has the additional condition that $\cS$ is closed under the multiplication by positive scalars. This condition is used to produce matrices in $\cS$ of spectral radius one and, therefore, is not essential in our context). We will consider two cases.

{\it Case 1}. Suppose that the idempotents $\{P_i\mid i=1,\dots,m\}$ span the entire space: $\span\{\range P_i\mid i=1,\dots,m\}=\bbC_n$. Applying a similarity to~$\cS$, we can write $P_j=E_{jj}\otimes I_k$ for each $j\in\{1,\dots,m\}$, where $I_k$ is the $k\times k$ identity matrix. In particular, this similarity makes sure $P_i=P_i^*$ for all $i$ and $\range P_i\perp\range P_j$ for $i\ne j$.

Consider, for each $i\in\{1,\dots,m\}$, the set $\cU_i=P_i\cS P_i|_{\range P_i}\setminus \{0\}$ as a set of of $k\times k$ matrices. Clearly, every member of $\cU_i$ is an invertible $k\times k$ matrix and $\cU_i$ is, in fact, a semigroup. Also, $\cU_i$ is bounded and norm closed, and the only idempotent matrix in $\cU_i$ is the identity matrix. By \cite[Lemma 3.1.6]{RR00}, each $\cU_i$ is a bounded group of matrices. Therefore, each $\cU_i$ is similar to a group of unitaries. Applying a simultaneous block-diagonal similarity to $\cS$, we may assume that each $\cU_i$ consists of unitary matrices. Notice that by Lemma~\ref{irreducibility-criterion}, each $\cU_i$ is also irreducible. Observe that we are done if $m=1$. Therefore, we will assume that $m\ge 2$.

For every pair $i,j\in\{1,\dots,m\}$, denote by $\cS_{ij}$ the set $P_i\cS P_j$. Clearly, each $\cS_{ij}$ is a subset of $\cS$ consisting of matrices of the form $E_{ij}\otimes T$ (where $T$ is a $k\times k$ matrix) and every nonzero element of $\cS_{ij}$ has rank~$k$. We will show that, first, $\cU_1=\cU_2=\dots=\cU_m=:\cU$ and, second, after a simultaneous similarity applied to~$\cS$, $\cS_{ij}=\{E_{ij}\otimes T\mid T\in\cU\cup\{0\}\}$ for all $i$ and~$j$.

To this end, fix an index $j\ne 1$, and, for notational simplicity, let $\cV=\cS_{11}$, $\cW=\cS_{jj}$, $\cX=\cS_{1j}$, and $\cY=\cS_{j1}$. The restrictions of members of these sets to the span of the ranges of $P_1$ and $P_j$ look as
$$
\cV=
\begin{bmatrix}
* & 0\\
0 & 0
\end{bmatrix},\ 
\cX=
\begin{bmatrix}
0 & *\\
0 & 0
\end{bmatrix},\ 
\cY=
\begin{bmatrix}
0 & 0\\
* & 0
\end{bmatrix},\ \mbox{and }
\cW=
\begin{bmatrix}
0 & 0\\
0 & *
\end{bmatrix}.
$$
Then the following inclusions, in particular, hold:
$$
\cV\cX\subseteq\cX,\quad
\cX\cY\subseteq\cV,\quad
$$
and
$$
\cY\cV\subseteq\cY,\quad
\cY\cX\subseteq\cW.
$$
We claim that for every $U\in\cV$ and a nonzero $Y\in\cY$, there is $X\in\cX$ such that $XY=U$. To see this, pick, by irreducibility, a matrix $X_1\in\cX$ such that $X_1Y\ne 0$ (just let $A=P_1$ and $B=P_jY$ in Lemma~\ref{irreducibility-criterion}). As observed above, $X_1Y\in\cV$, so we can write
$$
U_1:=X_1Y=\begin{bmatrix}
V_1 & 0\\
0 & 0
\end{bmatrix}.
$$
Since $V_1\ne 0$, we get $V_1\in\cU_1$, a unitary matrix. Therefore $V_1^*\in\cU_1$, so that $U_1^*\in\cV$ and $U_1^*U_1=P_i$. Define $U_2=UU_1^*\in\cV$. Again, by the observation made before, $U_2X_1\in\cX$. Put $X=U_2X_1$. Then $XY=U_2X_1Y=UU_1^*X_1Y=UU_1^*U_1=UP_1=U$. This proves the claim. 

Fix a nonzero $Y\in\cY$ and let $V\in\cV$ be arbitrary, nonzero. By the previous claim, there is $X\in\cX$ such that $XY=V$. Define $W=YX\in\cW$. Clearly, $\sigma(W)=\sigma(V)\ne\{0\}$, so that $W\ne 0$. Also
$$
YV=Y(XY)=(YX)Y=WY.
$$
Thus $(YV)^*=(WY)^*$. Multiplying these two identities and noticing that $W^*W=P_j$, $P_jY=Y$, $VV^*=P_1$, and $P_1Y^*=Y^*$, we get: 
$$
(YV)^*(YV)=(WY)^*(WY),
$$
or
$$
V^*Y^*YV=Y^*Y,
$$
so that 
$$
Y^*YV=VY^*Y.
$$
Since $V$ is an arbitrary member of $\cV$ and $\cU_1$ is irreducible, we conclude that $Y^*Y=\alpha P_1$, for some (nonzero) scalar $\alpha$. Similarly, $YY^*=\beta P_j$, for some scalar $\beta$. It follows that $\alpha=\beta\ne 0$. 

Repeating this for each index $j> 1$, we get: for each $j\in\{2,\dots,m\}$, there is a matrix $Y_j\in\cS_{j1}$ and a non-zero scalar $\alpha_j$ such that $Y_j^*Y_j=\alpha_j P_1$ and $Y_jY_j^*=\alpha_j P_j$.

Apply the similarity
$$
\begin{bmatrix}
I_k & 0 & \dots & 0\\
0 & \alpha_2 Y_2 & & \\
 & & \ddots & \\
0 & 0 & \dots &\alpha_mY_m
\end{bmatrix}
$$
to $\cS$. It is easy to see that, after this similarity, the matrix $E_{1j}\otimes I_k\in\cS_{1j}$, and the matrix $E_{j1}\otimes I_k\in\cS_{j1}$, for each $j\in\{1,\dots,m\}$. Then for all $i,j\in\{1,\dots,m\}$, we have $E_{ij}\otimes I_k=(E_{i1}\otimes I_k)\cdot(E_{1j}\otimes I_k)\in\cS_{ij}$, and this, clearly, implies that $E_{ij}\otimes T\in\cS_{ij}\iff E_{lr}\otimes T\in\cS_{lr}$, for all $i,j,l,r\in\{1,\dots,m\}$ and any $k\times k$ matrix $T$.

Let $T\in\cS$ be an arbitrary member of~$\cS$. Write $T=(T_{ij})_{i,j=1}^m$, where each $T_{ij}$ is a $k\times k$ matrix. Since $P_iTP_j\in\cS_{ij}$, it is clear that each $T_{ij}\in\cU\cup\{0\}$. We claim that each block row and each block column has at most one nonzero block. This, clearly, will imply that $T$ is a partial isometry. Suppose that for some $i,j_1$ and $j_2$, we have $T_{ij_1}\ne 0$ and $T_{ij_2}\ne 0$. Multiplying $T$ by $E_{ii}\otimes I_k$ on the left, if necessary, we may assume that $T_{kl}=0$ for all $l$ and all $k\ne i$. Also, multiplying by $E_{j_1i}\otimes I_k$ on the left, we may assume that $j_1=i$. It follows that the matrix $E_{ii}\otimes T_{ii}^*\in\cS$. Therefore, multiplying $T$ by this matrix on the left, we may assume that $T_{ii}=I_k$. It follows that $T$ is an idempotent which is different from $E_i$ and disjoint from it, which is impossible. The case of two nonzero blocks in a block column is similar.

{\it Case 2}. Suppose that the projections $P_1,\dots,P_m$ do not span the entire space. Let us show that this case cannot be realized. Indeed, repeating the argument of {\it Case~1}, we obtain that, relative to the decomposition $(\range P_1)\oplus\dots\oplus(\range P_m)\oplus H$, where $H$ is the orthogonal complement to the span of the ranges of projections $P_1,\dots, P_m$, each member $T\in\cS$ is written as
$$
T=\left[\begin{array}{ccc|c}
T_{11} & \dots & T_{1m} & X_1\\
\vdots & & \vdots & \vdots\\
T_{m1} & \dots & T_{mm} & X_{m}\\
\hline
Y_1 & \dots & Y_m & Z
\end{array}\right],
$$
where the blocks $T_{ij}$ ($1\le i,j\le m$) come from a fixed group of unitaries. Again, repeating the argument in the last part of {\it Case~1}, we obtain that if $X_1\ne 0$, then $T_{1j}=0$ for all~$j$. Multiplying by $P_1$ on the left, we may assume that $X_1$ is the only nonzero block in~$T$. Using Lemma~\ref{irreducibility-criterion}, we can find a matrix $Q\in\cS$ such that only the $(m+1,1)$-block of $Q$ is not zero and $QT\ne 0$. Applying Lemma~\ref{idempotents} to $QT$, we conclude that $\cS$ admits an idempotent $P$ such that $\range(P)\subseteq H$ and $\ker(P)\supseteq H^\perp$, contrary to the maximality of the family $\{P_1,\dots,P_m\}$. This finishes the proof of the equivalence \makebox{(i)$\iff$(ii)}.

\bigskip

Let us now prove the equivalence \makebox{(ii)$\iff$(iii)}. First, we will establish the easier implication \makebox{(ii)$\implies$(iii)}. Suppose that $\sigma(T)\subseteq\{0\}\cup\bbT$ for all $T\in\cS$. Using the trace considerations, as before, we conclude that $\cS$ is a bounded semigroup. That is, there is $C\ge 1$ such that $\norm{T}\le C$ for all $T\in\cS$. Let us establish the existence of the lower bound. If $T\in\cS$ is not nilpotent, then $r(T)=1$, so that $\norm{T}\ge r(T)=1$. If $T\ne 0$ is nilpotent, consider the ideal $\cJ=\cS T\cS$ in $\cS$. By \cite[Lemma 2.1.10]{RR00}, $\cJ$ is an irreducible semigroup and, hence, by the Levitzki's theorem (\cite{Lev31}, see also \cite[Theorem 2.1.7]{RR00}), $\cJ$ contains a non-nilpotent operator~$R$. We can write $R=ATB$ where $A,B\in\cS$. It follows that $1\le\norm{R}\le\norm{A}\cdot\norm{T}\cdot\norm{B}\le C^2\norm{T}$, so that $\norm{T}\ge \frac{1}{C^2}$.

\medskip

Finally, let us prove the implication \makebox{(iii)$\implies$(ii)}. We will generally follow the ideas of the proof of \cite[Lemma 2.4]{OR97} which shows that if every operator $T$ in an irreducible semigroup $\cS$ has spectral radius one, then $\sigma(T)\subseteq\{0\}\cup\bbT$ for all $T\in\cS$. However, we need to adapt that proof to our problem because the semigroup in \cite[Lemma 2.4]{OR97} does not have nilpotent elements and, in particular, does not have zero divisors.

Clearly, the boundedness of $\cS$ implies that $\sigma(T)\subseteq\bbD$ for every $T\in\cS$. Also, every non-nilpotent member of $\cS$ must have spectral radius one, for otherwise $\cS$ would contain elements of arbitrarily small nonzero norm. We need to prove that for each $T\in\cS$, $\sigma(T)\setminus\bbT\subseteq\{0\}$.

Suppose that this is not the case. Fix a matrix $T\in\cS$ of the minimal possible rank such that $T$ has a non-zero eigenvalue of modulus less than one. Write the Jordan form of $T$ as
$$
T=\begin{bmatrix}
U & 0 \\
0 & R
\end{bmatrix},
$$
where $\sigma(U)\subseteq\bbT$ and $r(R)<1$. Notice that $U\ne 0$ and $R\ne 0$. By Lemma~\ref{bounded-powers-unitary}, $U$ is a unitary diagonal matrix. Let $\lambda_1,\dots,\lambda_s$ be distinct eigenvalues of $U$ and $E_1,\dots,E_s$ be the correspondent spectral projections. Similarly, let $\mu_1,\dots,\mu_t$ be distinct eigenvalues of~$R$ and $F_1,\dots,F_t$ the correspondent spectral projections. Then we can write 
$$
\begin{bmatrix}
U & 0 \\ 0 & 0
\end{bmatrix}=\sum_{i=1}^s\lambda_iE_i\quad\mbox{and}\quad \begin{bmatrix}
0 & 0 \\ 0 & R
\end{bmatrix}=\sum_{i=1}^t\mu_i(F_i+N_i),
$$
where each $N_i$ is a nilpotent matrix commuting with~$F_i$. Observe that a sequence of powers of $T$ converges to an idempotent that we will denote by~$P_T$. This idempotent is written as
$$
P_T=\begin{bmatrix}
I_{\rk U} & 0 \\ 0 & 0
\end{bmatrix}
$$
and belongs to $\cS$ because $\cS$ is norm closed.

Since $P_T\in\cS$, the set $\cS_1=P_T\cS P_T$ forms a subsemigroup of~$\cS$. Clearly, the semigroup $\cS_2=\cS_1|_{\range P_T}$, considered as a set of $(\rk U)\times(\rk U)$ matrices, is an irreducible norm closed semigroup. Also, every member of $\cS_2$ has rank strictly smaller than $\rk(T)$, because $R\ne 0$. By the choice of $T$, the spectrum of every element of $\cS_2$ is contained in $\{0\}\cup\bbT$. Moreover, the idempotent elements of $\cS_2$ commute. Therefore, applying a similarity to $\cS$ and using the (already proved) equivalence \makebox{(i)$\iff$(ii)}, we may assume that $\cS_2$, as well as $\cS_1$, consists of partial isometries.

According to Theorem~\ref{part-isom-semigroups}, every matrix $A\in\cS_2$ is a block matrix whose blocks are elements of a fixed group of unitaries of size equal to the minimal nonzero rank of elements of $\cS_2$ which we will denote by~$k$. Moreover, at most one block in each block row and in each block column of $A$ is nonzero. In particular, this is true for the matrix $U$. Replacing $T$ with its power, we may assume that $U$ is block diagonal. Since $U$ is invertible, its every diagonal block is nonzero and, hence, is a unitary matrix. Furthermore, replacing $T$ with a higher power, we may also assume that $\norm{R}<\frac{1}{C^2}$ where $C=\sup\{\norm{S}\mid S\in\cS\}$. Notice that Theorem~\ref{part-isom-semigroups} also implies that the matrix
$$
P_0=\begin{bmatrix}
E_{11}\otimes I_k & 0\\
0 & 0
\end{bmatrix}
$$
belongs to~$\cS_1$.

Let $n_0\in\bbN$ be such that $N_i^{n_0}=0$ for all~$i$. Denote by $I$ the $n\times n$ identity matrix. Since $\cS$ is irreducible, Lemma~\ref{irreducibility-criterion} implies that there is $B\in\cS$ such that $(I-P_T)T^{n_0}(I-P_T)BP_0\ne 0$. Thus, again, by Lemma~\ref{irreducibility-criterion}, there exists $A\in\cS$ such that $P_0A(I-P_T)T^{n_0}(I-P_T)BP_0\ne 0$. Since $P_0\in\cS$, there is no loss of generality in assuming $P_0A=A$ and $BP_0=B$. It follows that, denoting the number of blocks in block rows of $\cS_2$ by~$m$, we may write $A$ and $B$ as
$$
A=\begin{bmatrix}
\widetilde A & A_{m+1} \\
0 & 0
\end{bmatrix}
\quad
\mbox{and}
\quad
B=\begin{bmatrix}
\widetilde B & 0 \\
B_{m+1} & 0
\end{bmatrix},
$$
where $\widetilde A$ and $\widetilde B$ are block matrices such that only the first block row of $\widetilde A$ and only the first block column of $\widetilde B$ may have nonzero blocks. Write the first block row of $\widetilde A$ as $\left[A_1,A_2,\dots,A_m\right]$, and the first block column of $\widetilde B$ as $\left[B_1,B_2,\dots,B_m\right]^T$. Notice that the condition $A(I-P_T)T^{n_0}(I-P_T)B\ne 0$ implies $A_{m+1}R^{n_0}B_{m+1}\ne 0$. Since $\norm{R}<\frac{1}{C^2}$, we have $0<\norm{A_{m+1}R^{n_0}B_{m+1}}<1$.

Define 
\begin{equation}\label{W_n}
W_n=AP_TT^nP_TB=
\begin{bmatrix}
\widetilde A U^n\widetilde B & 0 \\
0 & 0
\end{bmatrix}=
\sum_{i=1}^s\lambda_i^nAP_TE_iP_TB\in\cS
\end{equation}
and
\begin{equation}
\label{Q_n}
\begin{split}
Q_n
=AT^nB-W_n&
=\begin{bmatrix}
A_{m+1}R^nB_{m+1} & 0 \\
0 & 0
\end{bmatrix}=\\
&=\sum_{i=1}^t\mu_i^n\sum_{j=0}^{\min\{n,n_0-1\}}\left(\!\!
\begin{array}{c}
n\\
j
\end{array}\!\!
\right)A(I-P_T)F_iN_i^j(I-P_T)B.
\end{split}
\end{equation}
Since $W_n=P_0W_nP_0$ and $AT^nB=P_0AT^nBP_0$ for all~$n$, we may write $W_n$ and $AT^nB$ in the form
$$
W_n=\begin{bmatrix}
E_{11}\otimes V_n & 0\\
0 & 0
\end{bmatrix}\in\cS
\quad\mbox{and}\quad
AT^nB=\begin{bmatrix}
E_{11}\otimes Z_n & 0\\
0 & 0
\end{bmatrix}\in\cS,
$$
where each $V_n$ and $Z_n$ is either a unitary or zero. If $V_{n_0}$ were equal to zero, then we would have $\norm{Z_{n_0}}=\norm{E_{11}\otimes Z_{n_0}-E_{11}\otimes V_{n_0}}=\norm{A_{m+1}R^{n_0}B_{m+1}}\in(0,1)$, which is impossible. Therefore, $V_{n_0}$ is not zero.

Observe that if $U_1,U_2,\dots,U_m$ are the diagonal blocks of~$U$, then $V_{n}=A_1U^{n}_1B_1+A_2U^{n}_2B_2+\dots+A_mU^{n}_mB_m$ for every $n\in\bbN$. Only one of the matrices $\{A_1,A_2,\dots,A_m\}$ and only one of the matrices $\{B_1,B_2,\dots,B_m\}$ are nonzero, by the properties of~$\cS_2$. It follows that there exists $i_0\in\{1,\dots,m\}$ such that $A_{i_0}\ne 0$, $U^{n_0}_{i_0}\ne 0$, and $B_{i_0}\ne 0$. Hence, $A_{i_0}$, $U_{i_0}$, and $B_{i_0}$ are all unitaries. This implies that $V_n\ne 0$ for all $n$. Thus, each $V_n$ is necessarily a unitary. Also, it is clear that $\norm{R^n}\le\norm{R}$ for all~$n$, so that $\norm{A_{m+1}R^nB_{m+1}}<1$ for all~$n$. Therefore $\norm{Z_n}\ge\norm{V_n}-\norm{A_{m+1}R^nB_{m+1}}>0$. We conclude that $V_n$ and $Z_n$ are unitary matrices for all $n\in\bbN$.

This, in particular, implies that 
$$
W_n^*Q_n+Q_n^*W_n+Q_n^*Q_n=0.
$$
Fix an arbitrary vector $x$ from the range of~$P_0$. Then
\begin{equation}
\label{eq-1}
\langle Q_nx,W_nx\rangle+\langle W_nx,Q_nx\rangle+\langle Q_nx,Q_nx\rangle=0.
\end{equation}
Notice that, by the definition of $W_n$ and $Q_n$ (formulas~\eqref{W_n} and~\eqref{Q_n}), the left hand side of this equation can be viewed as a polynomial in $\alpha_1,\dots,\alpha_r$, where $\alpha_1,\dots,\alpha_r$ are all the distinct products of the form $\lambda_i\overline\mu_j$, $\mu_i\overline\lambda_j$, and $\mu_i\overline\mu_j$. To be more precise, the polynomial is written in the form
\begin{equation}
\label{polynomial}
\sum_i
\left[
\sum_{j=1}^{\min\{n,n_0-1\}}
\left(\!\!
\begin{array}{c}
n\\
j
\end{array}\!\!
\right)\gamma_{ij}
\right]\alpha_i^n
+
\sum_i
\left[
\sum_{j=1}^{\min\{n,n_0-1\}}
\left(\!\!
\begin{array}{c}
n\\
j
\end{array}\!\!
\right)^2\delta_{ij}
\right]\alpha_i^n,
\end{equation}
where $\gamma_{ij}$ and $\delta_{ij}$ denote some scalars independent on~$n$. Since the quantity
$$
\left(\!\!
\begin{array}{c}
n\\
j
\end{array}\!\!
\right)=\frac{1}{j!}\ n(n-1)\cdots(n-j+1)
$$
can be viewed as a polynomial in~$n$, we conclude that the coefficients in front of each $\alpha_i^n$ in the formula~\eqref{polynomial} are all polynomials in~$n$. Moreover, if $n\ge n_0$, the coefficients in these polynomials do not depend on~$n$. Hence, for $n\ge n_0$ the equation~\eqref{eq-1} transforms into
\begin{equation}
\label{eq-sum-poly}
\sum_{i}p_i(n)\alpha_i^n=0,\quad n\ge n_0,
\end{equation}
where $p_i$ are polynomials in one variable. By Lemma~\ref{polynomials-zero}, each $p_i$ is identically equal to zero. 
Consider the circle of the smallest radius $r$ on which there is a point of spectrum of~$T$. By the assumptions about the matrix~$T$, $0<r<1$. Clearly, $r^2$ is among $\alpha_i$, and $\alpha_i$ is equal to $r^2$ only if it is of the form $\mu_k\overline\mu_k$, for some~$k$. For all such~$i$, the equation~\eqref{eq-sum-poly} yields
%
%
$$
\norm{A(I-P_T)F_iN_i^j(I-P_T)Bx}=0,
$$
for all~$j$. Since $x\in\range(P_0)$ was taken arbitrary, this implies that
$$
A(I-P_T)F_iN_i^j(I-P_T)B=0
$$
for all $j$ and for all $i$ such that $\alpha_i=r^2$. This, however, shows that in the definition~\eqref{Q_n} of~$Q_n$, there are no summands corresponding to the eigenvalues of modulus~$r$. It follows, in particular, that if $r_1$ is the smallest radius of a circle on which there is a point of spectrum of $T$ of modulus in the interval $(r,1)$, then $\alpha_i$ can be equal to $r_1$ only if it is, again, of the form $\mu_k\overline\mu_k$, for some~$k$.

Repeating this argument inductively, we get: 
$$
A(I-P_T)F_iN_i^j(I-P_T)B=0
$$
for all $i$ and $j$. This clearly implies that $A(I-P_T)T^{n_0}(I-P_T)B=0$. This, however, contradicts the choice of $A$ and~$B$.
\end{proof}

We will finish the paper with the following example which shows that the assumption that the semigroup is norm closed is important in Theorem~\ref{main-result}. The main problem is that the condition that the idempotents in a semigroup commute may not survive after taking the norm closure.

\begin{example}\label{closure-important}
Let $t>0$ be an irrational real number. Consider the semigroup
$$
\cS=
\left\{
\begin{bmatrix}
e^{int} & 0\\
e^{imt} & 0
\end{bmatrix},
\begin{bmatrix}
0 & e^{int}\\
0 & e^{imt}
\end{bmatrix}
\mid n,m\in\bbN
\right\}.
$$
Obviously, $\cS$ is irreducible. Also, if $T\in\cS$ is arbitrary, then $\norm T=\sqrt 2$ and $\sigma(T)\in\{0\}\cup\bbT$. Since $\cS$ does not contain idempotents, the conditions~(ii) and~(iii) of Theorem~\ref{main-result} are satisfied. If $\cS$ were similar to a semigroup of partial isometries, so would be~$\overline\cS$. By Theorem~\ref{part-isom-semigroups}, the idempotents of $\overline\cS$ must commute which is not true since $\overline\cS$ contains the matrices
$$
\begin{bmatrix}
1 & 0\\
1 & 0
\end{bmatrix}
\quad\mbox{and}\quad
\begin{bmatrix}
0 & 1\\
0 & 1
\end{bmatrix}.
$$
So, the condition~(i) of Theorem~\ref{main-result} is violated.
\end{example}

\end{document}